\documentclass{article}

\usepackage{amsmath}
\usepackage{amsthm}
\usepackage{amssymb}
\usepackage{verbatim}
\usepackage{ifthen}
\usepackage{amsfonts}
\usepackage[mathscr]{euscript}   
\usepackage{xypic}      
\usepackage{graphicx}
\usepackage{textcomp}
\usepackage{yhmath}
\usepackage{bbm}

\newtheorem{theorem}{Theorem}[section]
\newtheorem{corollary}[theorem]{Corollary}
\newtheorem{lemma}[theorem]{Lemma}
\newtheorem{proposition}[theorem]{Proposition}

\newtheorem{question}[theorem]{Question}

\theoremstyle{definition}
\newtheorem{definition}[theorem]{Definition}

\theoremstyle{remark}
\newtheorem{remark}[theorem]{Remark}
\newtheorem{claim}[theorem]{Claim}

\numberwithin{equation}{section}

\newcommand{\eref}[1]{(\ref{#1})}

\newcommand{\scr}{\mathcal}

 \DeclareMathOperator{\End}{End}
 \DeclareMathOperator{\Hom}{Hom}

 \DeclareMathOperator{\Symb}{Symb}

\newcommand{\all}{\hspace{1em}\mbox{for all\ }}

\newcommand{\st}{\; | \;}

\newcommand{\CC}{\mathbb{C}}

\newcommand{\NN}{\mathbb{N}}

\newcommand{\half}{\frac{1}{2}}

\newcommand{\ip}[1]{\langle #1 \rangle}
\newcommand{\bigip}[1]{\left\langle #1 \right\rangle}

\newcommand{\slot}{\:\cdot\:}

\newcommand{\dual}{\dagger}

\newcommand{\pddiff}[3][.]{\ifthenelse{\equal{#1}{.}}
        {\frac{\partial^2 #2}{\partial #3^2}}
        {\frac{\partial^2 #1}{\partial #2 \partial #3}}}

\newcommand{\Lie}{\mathsf}
\newcommand{\lie}{\mathfrak}

\newcommand{\Univ}[2][.]{\ifthenelse{\equal{#1}{.}}{\mathcal{U}}{\mathcal{U}^{(#1)}}(\mathfrak{#2})}

 \DeclareMathOperator{\SL}{SL}
 
 \DeclareMathOperator{\SU}{SU}

\newcommand{\glx}{\mathfrak{gl}}
\newcommand{\slx}{\mathfrak{sl}}

\newcommand{\GTsO}[1]{
  \left( \begin{array}{cccccccccc}
  \multicolumn{2}{c}{#1} &
  \multicolumn{2}{c}{0} &
  \multicolumn{2}{c}{\cdots} &
  \multicolumn{2}{c}{0} &
  \multicolumn{2}{c}{\!\!-#1} \\
&\multicolumn{2}{c}{\;0\;}
&\multicolumn{4}{c}{\;0\,\cdots\,0\;}
&\multicolumn{2}{c}{\;0\;} \\
&&\multicolumn{2}{c}{\ddots}
&&&\multicolumn{2}{c}{\adots}\\
&&&\multicolumn{2}{c}{\;0}
&\multicolumn{2}{c}{0} \\
&&&&\multicolumn{2}{c}{0}
\end{array} \right)
}

\newcommand{\CX}{C(\scr{X})}
\newcommand{\CXS}{C(\XS)}

\newcommand{\FS}{\scr{F}_S}

\newcommand{\GF}{\scr{G}_\scrF}

\newcommand{\Kfin}{\Gamma_{\Lie{K}\mathrm{-fin}}}
\newcommand{\ES}{\scr{E}_S}

\newcommand{\LXE}[1][]{L^2(\scr{X};E_{#1})}
\newcommand{\CXE}[1][]{C(\scr{X};E_{#1})}

\newcommand{\ESXE}[1][]{\scr{E}_{S}(\scr{X};E_{#1})}

\newcommand{\LXSE}[1][]{L^2(\XS;E_{#1})}
\newcommand{\CXSE}[1][]{C(\XS;E_{#1})}
\newcommand{\CinftyXSE}[1][]{C^\infty(\XS;E_{#1})}
\newcommand{\KXSE}[1][] {\Kfin(\scr{X};E_{#1})}

\newcommand{\ipS}[1]{\ip{#1}_{\CXS}}

\newcommand{\scrK}[1][]{\mathcal{K}_{#1}}
\newcommand{\scrJ}[1][]{\mathcal{J}_{#1}}
\newcommand{\Kalpha}{\scrK[\alpha]}

\newcommand{\scrA}[1][]{\mathcal{A}_{#1}}
\newcommand{\Aalpha}{\scrA[\alpha]}
\newcommand{\scrKo}{\scrK[\emptyset]}

\newcommand{\Kp}{\scrK[p]}

\newcommand{\scrX}{\scr{X}}
\newcommand{\scrY}{\scr{Y}}
\newcommand{\scrV}{\scr{V}}
\newcommand{\scrF}{\scr{F}}
\newcommand{\XS}{\scrX_S}
\newcommand{\XT}{\scrX_T}

\newcommand{\KS}{\Lie{K}_S}
\newcommand{\KT}{\Lie{K}_T}

\newcommand{\irrep}[1]{\hat{\Lie{#1}}}
\newcommand{\Khat}{\hat{\Lie{K}}}
\newcommand{\KShat}{\hat{\Lie{K}}_S}
\newcommand{\KThat}{\hat{\Lie{K}}_T}
\newcommand{\Ki}{\Lie{K}^{(i)}}
\newcommand{\Kihat}{\hat{\Lie{K}}^{(i)}}

\newcommand{\trivS}{{\mathbbm{1}_S}}
\newcommand{\trivT}{{\mathbbm{1}_T}}

\newcommand{\PsiDOF}{\overline{\Psi^0_\scrF}}

\begin{document}

\title{Products of Longitudinal Pseudodifferential Operators on Flag Varieties}
\author{Robert Yuncken}

\maketitle

\begin{abstract}

Associated to each set $S$ of simple roots for $\SL(n,\CC)$ is an equivariant fibration $\scrX\to\XS$ of the space $\scrX$ of complete flags of $\CC^n$.  To each such fibration we associate an algebra $\scrJ[S]$ of operators on $L^2(\scrX)$ which  contains, in particular, the longitudinal pseudodifferential operators of negative order tangent to the fibres.  These form a lattice of operator ideals whose common intersection is the compact operators.  As a consequence, the product of fibrewise smoothing operators (for instance) along the fibres of two such fibrations, $\scrX\to\XS$ and $\scrX\to\XT$, is a compact operator if $S\cup T$ is the full set of simple roots. 

The construction uses noncommutative harmonic analysis, and hinges upon a representation theoretic property of subgroups of $\SU(n)$, which may be described as `essential orthogonality of subrepresentations'.

\end{abstract}


\section{Introduction}

Let $\scr{X} = \scr{X}_1 \times \scr{X}_2$ be a product of compact manifolds.  If $A_1$ and $A_2$ are longitudinal smoothing operators along the respective product fibrations, then their product  $A_1A_2$ is a smoothing operator on $X$.  More generally, if $A_1$ and $A_2$ are longitudinal pseudodifferential operators of negative order then their product, whilst not being a classical pseudodifferential operator, is a compact operator on $L^2(\scr{X})$.  In this article we extend the latter fact to a class of highly non-trivial multiply-fibred manifolds --- the complete flag varieties for $\CC^n$.  

The motivation for studying longitudinal pseudodifferential operators on flag varieties comes from the representation theory of semisimple groups, where they appear frequently.  For instance, the Kunze-Stein intertwining operators between principle series representations of $\SL(n,\CC)$ are of this form (see, {\em eg}, \cite{Knapp}).  This work originated from trying to understand the Bernstein-Gelfand-Gelfand resolution (\cite{BGG}; see also, {\em eg}, \cite{BasEast}) from the point of view of equivariant index theory.  In \cite{Bernstein-ICM}, Bernstein proposed a longitudinal Sobolev theory related to such operators.  However, as far as the present author is aware, certain desirable properties of this Sobolev theory seem to fail (see \cite[Chapter 5]{yuncken-thesis}).  In this light, the results presented here constitute a weaker analytic construction which, while far less powerful than a full Sobolev theory, is sufficient for some applications to index theory (see \cite{yuncken-sl3}).

Moreover, the main result here (Theorem \ref{meet}) applies to a broader class of operators than the longitudinal pseudodifferential operators.  This extra generality is useful in the index theoretic applications.

\bigskip

The main theorem is a consequence of a property of subgroups of $\SU(n)$, which may be paraphrased as `essential orthogonality of subrepresentations'.   Let $\pi$ be a unitary representation of a compact group $\Lie{K}$ on a Hilbert space $\scr{H}$.  If $\Lie{K}'$ is a closed subgroup of $\Lie{K}$, and $\sigma$ an irreducible representation of $\Lie{K}'$, then a vector $\xi\in\scr{H}$ is type $\sigma$ if the $\Lie{K}'$-subrepresentation of $\pi|_{\Lie{K}'}$ it generates is isomorphic to $\sigma$.

\begin{definition}
\label{essentially-orthotypical}
Two closed subgroups $\Lie{K}_1$ and $\Lie{K}_2$ of $\Lie{K}$ will be called {\em essentially orthotypical} if for any irreducible representations $\sigma_1$ of $\Lie{K}_1$ and $\sigma_2$  of $\Lie{K}_2$, and any $\epsilon>0$, there are only finitely many irreducible representations $\pi$ of $\Lie{K}$ which contain unit vectors $\xi_i$ of type $\sigma_i$  ($i=1,2$) such that $|\ip{\xi_1,\xi_2}| >\epsilon$.
\end{definition}

An equivalent formulation is that the product of the isotypical projections for $\sigma_1$ and $\sigma_2$ is compact on any unitary representation of $\Lie{K}$ with finite multiplicities.  (See Lemma \ref{asymptotic-orthogonality}.)

\begin{question}
\label{conjecture}
Is it true that $\Lie{K}_1$ and $\Lie{K}_2$ are essentially orthotypical whenever they generate $\Lie{K}$? 
\end{question}

Proposition \ref{asymptotic-orthogonality-true} confirms this for certain subgroups of $\SU(n)$.

\begin{remark}

Essential orthotypicality can be viewed as an strong version of Kazhdan's property $T$.  (Compact groups satisfy property $T$ trivially.)  If we consider $\Lie{K}_1 \cup \Lie{K}_2$ as a generating set for $\Lie{K}$, then the `almost invariant vectors' definition of property $T$ has the following consequence.  Let $\pi$ be an irreducible representation of $\Lie{K}$ on $V^\pi$. There exists $\delta>0$ such that if $|\ip{\xi_1,\xi_2}|>1-\delta$ for some unit vectors $\xi_1$ and $\xi_2$ in $V^\pi$ fixed by $\Lie{K}_1$ and $\Lie{K}_2$, respectively, then $\pi$ is the trivial representaion for $\Lie{K}$. 

On the other hand, essential orthotypicality says that {\em for any} $\epsilon>0$, the condition $|\ip{\xi_1,\xi_2}|>\epsilon$ implies that $\pi$ belongs to some finite set of irreducibles of $\Lie{K}$.


\end{remark}

\bigskip

In a different direction, the results presented here suggest obvious questions about longitudinal pseudodifferential operators on multiply foliated manifolds.  Suppose $\scrX$ is a compact manifold which admits two foliations $\scrF_1$ and $\scrF_2$ with compact leaves.  Suppose further that the tangent bundles to the foliations, $T\scrF_1$ and $T\scrF_2$, generate a distribution in $T\scrX$ which is totally non-integrable\footnote{One could weaken this assumption further by asking merely that the vector fields tangent to the two foliations generate all vector fields on $\scrX$ as a Lie algebra}.
\begin{enumerate}
\item If $A_i$ is a longitudinal smoothing operator along the leaves of $\scrF_i$ ($i=1,2$), is $A_1A_2$ a smoothing operator on $\scrX$?
\item Is $A_1A_2$ a compact operator on $\scrX$?
\end{enumerate}
The answer to {(i)} is no.  We suspect the answer to {(ii)} is yes.  However, the level of generality in these questions is greater than is necessary for the representation theoretic applications we have in mind.  Furthermore, the symmetry present for flag varieties allows us to take a noncommutative harmonic analysis approach to these questions, and this allows for the wider class of operators alluded to earlier.

\section{Longitudinal pseudodifferential operators on a fibre bundle}
\label{longitudinal-PsiDO}

Let $\scrX\stackrel{q}{\to}\scrY$ be a smooth fibre bundle.  The fibration yields a foliation of $\scrX$, which we will denote by $\scrF$.  Let $E$ be a vector bundle over $\scrX$.  The set of longitudinal pseudodifferential operators of order $p$ on $E$, tangent to $\scrF$, will be denoted by $\Psi^p_\scrF(E)$.  Most of the following background on longitudinal pseudodifferential operators can be found in \cite{MS-GAFS}.

Put a Riemannian metric on $\scrX$ and Hermitian metric on $E$, so that we can define the $L^2$-sections of $E$.  
The order zero longitudinal pseudodifferential operators are bounded on $L^2(\scrX;E)$.  Let $S^*\scrF$ be the cosphere bundle of the foliation.  The tangential principal symbol map
$$
  \Symb_0 : \Psi^0_\scrF(E) \to C(S^*\scrF, \End(E))
$$
extends continuously to the operator-norm closure of $\Psi^0_\scrF(E)$.  Moreover, there is a short exact sequence of $C^*$-algebras
$$\xymatrix{
 0 \ar[r] & \overline{\Psi^{-1}_\scrF}(E) \ar[r]
  &\PsiDOF(E) \ar[r]^-{\Symb_0} & C(S^*\scrF, \End(E)) \ar[r] & 0. }
$$
(The kernel $\overline{\Psi^{-1}_\scrF}(E)$ is equal to $C^*_r(\GF;E)$, the $C^*$-algebra of the foliation groupoid associated to $\scrF$, although we shall not need this here.)  In fact, $ \overline{\Psi^{-1}_\scrF}(E) =  \overline{\Psi^{-p}_\scrF}(E)$ for any $-\infty\leq-p<0$.

The ideal $\overline{\Psi^{-1}_\scrF}(E)$ is much simplified in the case where $\scrF$ comes from a fibre bundle.  One can define an inner product on continuous sections of $E$ with values in $C(\scrY)$ by $L^2$-integration along the fibres:
\begin{equation}
 \label{Hilbert-module-inner-product}
 \ip{s_1,s_2}_{C(\scrY)}(y) = \int_{q^{-1}(y)} \ip{s_1(x), s_2(x)}_x \, d\mathrm{Vol}_{q^{-1}(y)} (x),
\end{equation}
for $s_1,s_2\in\CXE$.  Thus, $\CXE$ completes to a Hilbert $C(\scrY)$-module, which we denote by $\scr{E}_\scrF(\scrX;E)$.  The following fact is certainly well-known, although we are not aware of a specific reference.  We therefore provide a brief proof.

\begin{proposition}
\label{fibration-groupoid-algebra}

The algebra $\overline{\Psi^{-\infty}_\scrF}(E)$ is isomorphic to the algebra of compact Hilbert module operators $\scr{K}(\scr{E}_\scrF(\scrX;E))$.

\end{proposition}

\begin{proof}[Proof (sketch)]

Since $\scrX$ is compact, the choice of metrics on $\scrX$ and $E$ will not affect the algebras.  If the fibration is trivial ($\scrX = \scrY \times \scrV$) and the bundle $E$ is the pullback of a bundle on the fibre $\scrV$  then the result is a bundle version of the standard fact that the completion of the smoothing operators on a compact manifold is the compact operators.  To generalize this, observe that the bundle $E\to\scrX$ is locally of the above product form.  Use a partition of unity subordinate to a finite trivializing cover of $\scrY$ to show that the two algebras of the proposition are each included in the other with bounded change in norm.

\end{proof}

\section{Semisimple groups and homogeneous spaces}

\label{section-LG}

We will fix the following notation throughout this paper.   Let $\Lie{K}$ be a compact semisimple Lie group, with Lie algebra $\lie{k}$.  Fix a maximal torus $\Lie{T}\subseteq \Lie{K}$, with Lie algebra $\lie{t}$.  Let $R\subset\lie{t}^\dual$ denote the root system for $\Lie{K}$, and fix a choice of simple roots $\Sigma = \{\alpha_1, \ldots, \alpha_n\}$.  Let $R^+$ be the positive roots.  Let $\Lambda_R$ and $\Lambda_W$ denote the root and weight lattices, respectively.  

We now associate to each subset $S\subseteq\Sigma$ of simple roots a reductive subgroup of $\Lie{K}$ as follows.  Let $\lie{g}=\lie{k}_\CC$ be the complexified Lie algebra, with Cartan subalgebra $\lie{h}=\lie{t}_\CC$.  Let $\ip{S}$ denote the set of roots of $\Lie{K}$ which are linear combinations of roots in $S$.  Define
$$\lie{k}_S = \lie{k} \cap \left( \lie{h} \oplus \bigoplus_{\mu\in \ip{S}} \lie{g}_\mu \right),$$
which is a block-diagonal Lie subalgebra of $\lie{k}$.  Let $\KS$ be the corresponding subgroup.  In the terminology of complex semisimple groups, this is the maximal compact subgroup of the reductive part $\Lie{M}_S\Lie{A}_S$ in the Langlands decomposition of the parabolic subgroup $\Lie{M}_S\Lie{A}_S \Lie{N}_S$ associated to $S\subseteq\Sigma$.  So, for instance, if $\Lie{K}=\SU(5)$, and $S=\{\alpha_1,\alpha_2,\alpha_4\}\subseteq\Sigma$, then
$$
 \KS = \left\{ 
 \renewcommand{\baselinestretch}{1}\small\normalsize
     \left( \begin{array}{ccc}
     &&\\ 
     \multicolumn{2}{c}{\raisebox{1.5ex}[0pt]{$A$}} &  \raisebox{1.5ex}[0pt]{$0$} \\  
     \multicolumn{2}{c}{0}&B
     \end{array} \right)
 \renewcommand{\baselinestretch}{2}\small\normalsize 
 : A\in\mathrm{U}(3), ~B\in\mathrm{U}(2), ~(\det A)(\det B) = 1 \right\}.
$$
Note that $\Lie{K}_\emptyset = \Lie{T}$.

We use $\XS$ to denote the generalized flag variety $\Lie{K}/\KS$.  The space of complete flags is $\scrX_\emptyset = \Lie{K}/T$, which we will denote simply by $\scr{X}$.  For each $S\subseteq\Sigma$, the quotient map $\scr{X}\stackrel{q_S}{\to}\XS$ defines a fibration of $\scrX$ with fibres $\Lie{K}_S/\Lie{T}$.

\section{Harmonic decompositions}

We begin with some generalities.  Let $\Lie{K}$ be a compact group, and $\Lie{H}$ a closed subgroup.  Let $U$ be a unitary representation of $\Lie{K}$ on a Hilbert space $\scr{H}$.  If $\sigma\in\irrep{H}$ is an irreducible representation of $\Lie{H}$, we let $p_\sigma$ denote the projection onto the $\sigma$-isotypical subspace of $\scr{H}$ (restricting the representation of $\Lie{K}$ to $\Lie{H}$).  This can be written explicitly as
\begin{equation}
\label{integral-formula}
 p_\sigma = \dim\sigma.\int_{\Lie{H}} \overline{\chi_\sigma(h)}U(h)\,dh,
\end{equation}
where $\chi_\sigma$ is the character of $\sigma$.  If $F\subseteq\irrep{H}$ is a collection of irreducible representations, then we put $P_F = \sum_{\sigma\in F} p_\sigma$.

\begin{lemma}
\label{commuting-projections}
Let $\Lie{H}_1$, $\Lie{H}_2$ be closed subgroups of $\Lie{K}$, and let $\sigma\in\irrep{H}_1$, $\tau\in\irrep{H}_2$.   
\begin{enumerate}
\item
  If $\Lie{H}_1$ and $\Lie{H}_2$ commute, then $p_\sigma$ and $p_\tau$ commute.  
\item
  If $\Lie{H}_1 \leq \Lie{H}_2$, then $p_\sigma$ and $p_\tau$ commute.  
\end{enumerate}
\end{lemma}

\begin{proof}
Making the change of variables $h_2 \mapsto h_1 h_2 h_1^{-1}$ in the following integral, we get
\begin{eqnarray*}
 p_\sigma p_\tau 
  &=& \dim\sigma .\dim\tau .\int_{h_1\in\Lie{H}_1} \int_{h_2\in\Lie{H}_2} \overline{\chi_\sigma(h_1)\chi_\tau(h_2)}U(h_1 h_2)\,dh_1 dh_2 \\
  &=& \dim\sigma .\dim\tau .\int_{h_1\in\Lie{H}_1} \int_{h_2\in\Lie{H}_2} \overline{\chi_\sigma(h_1)\chi_\tau(h_1h_2h_1^{-1})}U(h_2 h_1)\,dh_1 dh_2 .
\end{eqnarray*}
In either of the cases considered, we have $\chi_\tau(h_1h_2h_1^{-1}) = \chi_\tau(h_2)$, so the latter integral equals $p_\tau p_\sigma$.
\end{proof}

Now we specialize to the case of $\Lie{K}$ being compact semisimple and $\Lie{H}=\KS$, for some $S\subseteq\Sigma$.  

Consider first the case of $\Lie{K}_\emptyset = \Lie{T}$.  The irreducible representations of $\Lie{T}$ correspond to the weights $\mu$ of $\Lie{K}$, via the exponential map.  The corresponding harmonic projection --- which we will denote by $p_\mu$ rather than the cumbersome $p_{(e^{i\mu})}$ --- is the projection onto the $\mu$-weight space of a $\Lie{K}$-representation.

More generally, for any $S\subseteq\Sigma$, the family of projections $\{p_\sigma : \sigma\in\KShat \}$ give an orthogonal decomposition of any unitary representation space of $\Lie{K}$.   We wish to slightly enlarge the class of spaces which admit such harmonic decompositions.

\begin{definition}
Any direct sum of weight spaces $H=\bigoplus_i p_{\mu_i} \scr{H}_i$ (where $\scr{H}_i$ are Hilbert spaces with unitary $\Lie{K}$-representations, and $\mu_i\in\Lambda_W$) will be referred to as a {\em harmonic $\Lie{K}$-space}.
\end{definition}

By Lemma \ref{commuting-projections}, projections $p_\sigma$ (with $\sigma\in\KShat$) and $p_\tau$ (with $\tau\in\KThat$) commute if $T \subseteq S\subseteq \Sigma$.  In particular, the weight-space projections $p_\mu$ (with $\mu\in\Lambda_W$) commute with all of the harmonic projections $p_\sigma$.  Thus, for each $S\subseteq\Sigma$, the projections $\{p_\sigma : \sigma\in\KShat \}$ define an orthogonal decomposition of any harmonic $\Lie{K}$-space.

\section{Homogeneous vector bundles}

The key example of a harmonic $K$-space is the section space of a $\Lie{K}$-homogeneous vector bundle over the flag variety $\scr{X}$.  To this end, let us fix some notation.

Firstly, when working with harmonic projections $p_\sigma$ on $L^2(\Lie{K})$, we will {\em always} take them to be defined with respect to the {\em right} regular representation of $\Lie{K}$.

If $\sigma$ is a finite dimensional representation of any group, we will always denote its representation space by $V^\sigma$.  The contragredient representation will be denoted $\sigma^\dual$, acting on the dual space $V^{\sigma\dual}$.

If $\sigma$ is a finite dimensional representation of $\KS$, let $E_\sigma = \Lie{K}\times_{\KS}V^\sigma$ denote the $\Lie{K}$-homogeneous vector bundle over $\XS$ induced from $\sigma$.  Thus, the continuous sections of $E_\sigma$ are identified with
\begin{multline}
\label{equivariance}
 C(\XS;E_\sigma) = \{ s: \Lie{K} \to V^\sigma \st \text{$s$ is continuous and}\\
     s(kh) = \sigma(h^{-1})\,s(k) \text{ for all } k\in\Lie{K}, h\in\KS \}
\end{multline}

In the case of $\Lie{K}_\emptyset=\Lie{T}$, we will use weights $\mu$ in the notation, rather than their corresponding characters $e^{i\mu}$.  Thus,
\begin{eqnarray}
  \CXE[\mu] &=& \{ s\in C(\Lie{K}) \st  s(kt) = e^{i\mu}(t^{-1})\,s(k) \text{ for all } k\in\Lie{K}, t\in\Lie{T} \} 
    \nonumber \\
    &=& p_{-\mu} C(\Lie{K}). \label{mu-equivariance}
\end{eqnarray}
Hence $\LXE[\mu]=p_{-\mu}L^2(\Lie{K})$ is a harmonic $\Lie{K}$-space.  Moreover, any $\Lie{K}$-invariant vector bundle $E$ over $\scr{X}$ decomposes equivariantly into homogeneous line bundles, so that $\LXE$ is a harmonic $\Lie{K}$-space.

If $s_1, s_2\in\CXE[\mu]$, then $\overline{s_1(k)} s_2(k)$ is constant on right $\Lie{T}$-cosets, and this defines the $C(\scrX)$-valued inner product of sections, which in turn defines the Hermitian metric on $E_\mu$.  More generally, the $C(\XS)$-valued inner product on $\CXE[\mu]$ of formula \eref{Hilbert-module-inner-product} can be written as
\begin{eqnarray*}
  \ipS{s_1,s_2}(k) &=& \int_{h\in\KS} \overline{s_1(kh)} s_2(kh) \,dh \\
    &=& (p_{\trivS} (\overline{s_1} s_2) )(k),
\end{eqnarray*}
where $\trivS$ is the trivial representation of $\KS$.  The resulting Hilbert $\CXS$-module will be denoted $\ESXE[\mu]$.  Note the extreme cases $\scr{E}_\emptyset(\scrX;E_\mu) = \CXE[\mu]$ and $\scr{E}_\Sigma(\scrX;E_\mu) = \LXE[\mu]$.

\medskip
Before moving on to the central definitions of this paper, we mention one useful technical fact.  Let us extend the above $\CXS$-valued inner product to all of $\CX$, by the formula
\begin{equation}
\label{ES-pairing}
  \ipS{f_1,f_2} = p_{\trivS} (\overline{f_1}f_2) \qquad\qquad(f_1, f_2 \in \CX).
\end{equation}
Denote the completion of $\CX$ in this inner product by $\scr{E}_S(\Lie{K})$.
For $\sigma\in\KShat$, the projection $p_\sigma$ is $\CXS$-linear.  It is also adjointable (self-adjoint) since
\begin{eqnarray*}
  \ipS{p_\sigma f_1,f_2}(g) &=& \int_{h,k\in\KS} \chi_\sigma(h) f_1(gkh) f_2(gk) \,dk\,dh \\
    &=& \int_{h,k\in\KS} \overline{\chi_\sigma(h')} f_1(gk') f_2(gk'h') \,dk'\,dh' \\
    &=& \ipS{f_1,p_\sigma f_2}(g),
\end{eqnarray*}
by making the change of variables $k'=kh$, $h'=h^{-1}$.

\begin{lemma}
\label{finite-generation}
Let $S\subseteq\Sigma$, and let $\sigma\in\KShat$.  Then $p_\sigma C(\Lie{K})$ is a finitely generated projective Hilbert $C(\XS)$-module.
\end{lemma}

\begin{proof}

Recall that  $ \CXSE[\sigma^\dual] $ is a space of $V^{\sigma\dual}$-valued functions on $\Lie{K}$ (see Equation \eref{equivariance}).
The natural $\CXS$-valued inner product of sections $s,t\in\CXSE[\sigma^\dual]$ is given by
$$
  \ipS{s,t}(k) = \ip{s(k),t(k)}_{V^{\sigma\dual}}  \qquad\qquad (k\in\Lie{K}).
$$
We claim that there is an isomorphism of Hilbert $\CXS$-modules
\begin{eqnarray*}
   \Phi:\CXSE[\sigma^\dual] \otimes V^\sigma &\to& p_\sigma \scr{E}_S(\Lie{K}) \\
   s\otimes v \;\;\;&\mapsto& (s(\slot),v).
\end{eqnarray*}
Note that the image of $\Phi$ consists of continuous functions on $\Lie{K}$, so this will prove both that $p_\sigma \scr{E}_S(\Lie{K}) = p_\sigma C(\Lie{K})$ and that it is finitely generated projective.

We appeal to the well-known Peter-Weyl decomposition of $\LXSE[\sigma^\dual]$:
$$
  \LXSE[\sigma^\dual] \cong \bigoplus_{\pi\in\irrep{K}} V^{\pi\dual} \otimes \Hom_{\KS}(V^\sigma, V^\pi)
$$
In this picture, $\Phi$ is obtained by applying the isomorphisms
\begin{eqnarray*}
  V^{\pi\dual} \otimes \Hom_{\KS}(V^\sigma, V^\pi) \otimes V^\sigma &\to& V^{\pi\dual} \otimes p_\sigma V^\pi \\
  \eta^\dual \otimes A \otimes v  \;\;\; &\mapsto& \eta^\dual \otimes Av.
\end{eqnarray*}
Since Peter-Weyl gives $p_\sigma L^2(\Lie{K}) \cong \bigoplus_{\pi\in\irrep{K}} V^{\pi\dual} \otimes p_\sigma V^\pi$, we see that $\Phi$ is well-defined and has dense range.
It is clearly $\CXS$-linear.  Finally, given $s\otimes v$ and $t\otimes w$ in the domain of $\Phi$,
$$
  \ipS{ \Phi(s\otimes v), \Phi(t\otimes w)} (k) 
     = \int_{h\in\KS} \overline{ (s(k),\sigma(h)v)}\, (t(k),\sigma(h)  w) \, dh.
$$
But the map $V^{\sigma\dual}\otimes V^\sigma \to L^2(\KS); v^\dual\otimes v \mapsto (v^\dual , \sigma(\slot)v)$ is an isometry, up to a factor of $(\dim V^\sigma)^\half$, so the above integral is a fixed scalar multiple of
$$
  \ip{s(k)\otimes v, t(k)\otimes w} = \ipS{s\otimes v, t\otimes w}.
$$

\end{proof}

\section{$C^*$-algebras associated to the fibrations}

\begin{definition}
Fix $S\subseteq\Sigma$. Let $H_1$ and $H_2$ be harmonic $\Lie{K}$-spaces, and let $A:H_1\to H_2$ be a bounded linear map between them.  For each $\sigma, \tau\in\hat{\Lie{K}}_S$, put $A_{\sigma\tau}=p_\sigma A p_\tau$, so that $(A_{\sigma\tau})_{\sigma,\tau\in\hat{\Lie{K}}_S}$ is the matrix of $A$ with respect to the $\KS$-harmonic decomposition.  Say $A$ is
\begin{enumerate}
\item \emph{$S$-harmonically finite} if all but finitely many matrix entries $A_{\sigma\tau}$ are zero,
\item \emph{$S$-harmonically proper} if the matrix $(A_{\sigma\tau})$ is row- and column-finite, {\em ie}, for each fixed $\sigma$ there are only finitely many $\tau$ with $A_{\sigma\tau}$ or $A_{\tau\sigma}$ nonzero.
\end{enumerate}
\end{definition}

If $H_1=H_2=H$, the set of $S$-harmonically proper operators is an algebra, and the $S$-harmonically finite operators form an ideal in that algebra.  It is natural to close these in operator-norm to obtain a $C^*$-algebra and ideal.

\begin{definition}

For any $S\subseteq\Sigma$, let $\scrA[S](H_1,H_2)$ (respectively, $\scrK[S](H_1,H_2)$) denote the operator-norm closure of the $S$-harmonically proper operators (respectively $S$-harmonically finite operators) from $H_1$ to $H_2$.   If $H_1=H_2=H$, we will write $\scrA[S](H)$ for $\scrA[S](H,H)$ and $\scrK[S](H)$ for $\scrK[S](H,H)$.

\end{definition}

It is notationally convenient to think of $\scrA[S]$ and $\scrK[S]$ as $C^*$-categories, whose objects are harmonic $\Lie{K}$-spaces and whose morphism sets are given by the definition above.  However, it is worth remarking that we shall need none of the technicalities of $C^*$-categories.  This simply allows us to write $A\in\scrA[S]$ or $A\in\scrK[S]$, with the domain and target spaces implied by the definition of $A$.

\medskip

Fix $S\subseteq\Sigma$.  Let us fix an enumeration of the irreducible representations of $\Lie{K}_S$ as $\{\sigma_0, \sigma_1,\sigma_2,\ldots\}$, with $\sigma_0$ being the trivial representation.  Let $F_j=\{\sigma_i \st 0\leq i \leq j\}\subseteq\KShat$.  Recall that $P_{F_j}$ denotes the projection $\sum_{\sigma\in F_j}p_{\sigma}$.  

\begin{lemma}
\label{KS-equivalent-defns}
 
Let $K:H_1 \to H_2$ be a
bounded linear map between harmonic $\Lie{K}$-spaces.  The following are equivalent:
\begin{enumerate}

\item $K\in\scrK[S]$,

\item $P_{F_j}^\perp K \to 0$ and $K P_{F_j}^\perp \to 0$ in
norm as $j\to\infty$,

\item $P_{F_j} K P_{F_j} \to K$ in norm as $j\to\infty$.

\end{enumerate}

\end{lemma}

\begin{proof}

For {\em(i)$\Rightarrow$(ii)},  note that {\em(ii)} is immediate if $K$ is $S$-harmonically finite, and hence holds for all $K\in\scrK[S]$ by density.  The
implications {\em(ii)}$\Rightarrow${\em(iii)} and
{\em(iii)}$\Rightarrow${\em(i)} are straightforward.

\end{proof}

\begin{lemma}
\label{AS-equivalent-defns}
For a bounded linear map $A:H_1\to H_2$ between harmonic $\Lie{K}$-spaces, the following are equivalent:
\begin{enumerate}

\item $A\in \scrA[S]$,

\item For any $k\in\NN$, $P^\perp_{F_j} A P_{F_k} \to 0$ and $P_{F_k} A P^\perp_{F_j} \to 0$ in
norm as $j\to\infty$,

\item $A$ is a two-sided multiplier of $\scrK[S]$, ie, $AK\in\scrK[S]$ for all right-composable $K\in\scrK[S]$ and $KA\in\scrK[S]$ for all left-composable $K\in\scrK[S]$.

\end{enumerate}

\end{lemma}

\begin{remark}  
Here, {\em left-} and {\em right-composable} mean that the appropriate domain and target spaces agree.
\end{remark}

\begin{proof}

{\em(i)$\Rightarrow$(ii)}:  If $A$ is $S$-harmonically proper then {\em(ii)} is
immediate, so by density, {\em (ii)} holds for all $A\in\scr{A}_i$.

{\em(ii)$\Rightarrow$(iii)}:  Suppose $A$ satisfies {\em(ii)}.  If $K$ is $S$-harmonically finite and left-composable with $A$ then $KA$ satisfies {\em (ii)} of
Lemma \ref{KS-equivalent-defns}, so $KA\in\scrK[S]$.  Similarly, $AK\in\scrK[S]$ for right-composable  $S$-harmonically finite $K$.  Thus, {\em(iii)} follows by the density of $S$-harmonically finite operators in $\scrK[S]$.

{\em(iii)$\Rightarrow$(i)}:  Let $A$ be a multiplier of $\scrK[S]$.  Let $\epsilon>0$.  Starting with $B_0=A$, we
will construct a sequence $(B_k)$ of multipliers of $\scrK[S]$ such that
\begin{equation}
\label{small-difference}
  \|B_{k+1}-B_{k}\|<\epsilon.2^{-k-1},
\end{equation}
as well as a strictly increasing sequence $a_0, a_1, a_2,\ldots \in \NN$ such
that
\begin{equation}
\label{partially-proper-1}
 P_{F_{a_j}}^\perp B_k P_{F_j} =0 \qquad \text{for all }0\leq j < k
\end{equation}
and
\begin{equation}
\label{partially-proper-2}
  P_{F_j} B_k P_{F_{a_j}}^\perp=0 \qquad \text{for all }0\leq j < k.
\end{equation}
The norm-limit of these $B_k$ will be within $\epsilon$ of $A$ (by \eref{small-difference}) and will be $S$-harmonically proper (by \eref{partially-proper-1} and \eref{partially-proper-2}).

Suppose, then, that we have defined $B_{k}$.  Both $B_{k} P_{F_k}$ and
$P_{F_k} B_{k}$ are in $\scrK[S]$ by assumption, so by Lemma
\ref{KS-equivalent-defns} there is an integer $a_k$ (without loss of generality, larger than
$a_{k-1}$) such that the operators
$$C_k=P_{F_{a_k}}^\perp B_{k} P_{F_k}$$
and
$$D_k=P_{F_k} B_{k}  P_{F_{a_k}}^\perp,$$
have norm less than $\epsilon.2^{-k-2}$.  Now put
$$B_{k+1} = B_{k} - C_k - D_k.$$
It is clear that \eref{small-difference} is satisfied.  Since all isotypical projections for $\KS$ commute, 
\eref{partially-proper-1} and \eref{partially-proper-2} hold for $B_{k+1}$ with  $0\leq j < k$.  Finally, (noting that $a_k\geq k$)
\begin{equation*}
  P_{F_{a_k}}^\perp B_{k+1} P_{F_k} 
     = P_{F_{a_k}}^\perp B_{k} P_{F_k} -P_{F_{a_k}}^\perp B_{k} P_{F_k} -0 \\
     =0,
\end{equation*}
and $ P_{F_{a_k}}^\perp B_{k+1} P_{F_k}=0$ similarly.

\end{proof}

\begin{lemma}
\label{projections-in-AS}

Let $S,T\subseteq\Sigma$, and suppose $S\subseteq T$ or $S\supseteq T$.  Then $p_\sigma\in\scrA[T]$ for any $\sigma\in\KShat$.

\end{lemma}

\begin{proof}
By Lemma \ref{commuting-projections}, $p_\sigma$ commutes with the $\KT$-isotypical projections, and so preserves $T$-spectral finiteness.
\end{proof}

\begin{remark}
 In fact, $p_\sigma\in\scrA[T]$ for any $S,T\subseteq\Sigma$, although this is not obvious yet --- see the proof of Corollary \ref{projections-in-JS}.
 \end{remark}

\begin{definition}
A harmonic $\Lie{K}$-space $H$ will be called {\em finite multiplicity} if, for each $\pi\in\Khat$, $p_\pi H$ is finite-dimensional.
\end{definition}

The right regular representation is a finite multiplicity harmonic $\Lie{K}$-space by the Peter-Weyl Theorem.  Thus $\LXE[\mu] = p_\mu L^2(\Lie{K})$ is finite multiplicity for each $\mu$, and so is the $L^2$-section space of any finite dimensional $K$-homogeneous bundle over $\scrX$.

\begin{lemma}

If $H_1$ and $H_2$ are finite multiplicity harmonic $\Lie{K}$-spaces, then $\scrK[\Sigma](H_1,H_2)$ is the space of compact operators from $H_1$ to $H_2$.

\end{lemma}

\begin{proof}

This follows from Lemma \ref{KS-equivalent-defns} {\em(iii)}.

\end{proof}


\section{Multiplication operators}

\begin{lemma}
\label{mult-ops-in-AS}

Let $f\in C(\Lie{K})$. The operator $M_f$ of multiplication by $f$ belongs to $\scrA[S](L^2(\Lie{K}))$ for any $S\subseteq\Sigma$.  

\end{lemma}

\begin{proof}

In short, after applying the Peter-Weyl Isomorphism, multiplication of functions transforms to tensor product of representations.  Since the tensor product of two irreducible representations of $\KS$ decomposes again into finitely many irreducibles, this operation is $S$-harmonically proper.  We now make this precise.

Suppose first that $f$ is a matrix unit, that is, for some $\pi\in\Khat$ and $v\in V^\pi$, $w^\dual\in V^{\pi\dual}$
\begin{equation}
\label{matrix-unit1}
  f(k) = (w^\dual, \pi(k) v).
\end{equation}
Suppose moreover that $v$ is isotypical for $\KS$ --- specifically, $v\in p_{\tau} V^\pi$ for some $\tau\in\KShat$.

Consider an arbitary irreducible $\sigma\in\KShat$.  Let $s\in L^2(\Lie{K})$.  If $s$ is itself a matrix unit,
\begin{equation}
\label{matrix-unit2}
  s(k) = (\eta^\dual, \rho(k) \xi), 
\end{equation}
for some $\rho\in\Khat$ and $\xi\in V^{\rho}$, $\eta^\dual\in V^{\rho\dual}$, then 
$$
  (p_\sigma s)(k)= (\eta^\dual, \rho(k)\, p_\sigma\xi)
$$
The product of the matrix units \eref{matrix-unit1} and \eref{matrix-unit2} is
$$
  M_f(p_\sigma s)(k) = \left(w^\dual\otimes \eta^\dual, (\pi\!\otimes\!\rho)(k)\; (v\!\otimes\! p_\sigma\xi)\right).
$$
The vector $v\!\otimes\! p_\sigma\xi$ lies in a $\KS$-subrepresentation of $\pi\!\otimes\!\rho$ isomorphic to $\tau\!\otimes\!\sigma$, which decomposes into a finite set $F$ of $\KS$-types.  Thus, for any $s\in L^2(\Lie{K})$, $P_F^\perp M_f \,p_\sigma s=0$.

The adjoint of multiplication by $f$ is multiplication by $\overline{f}$, which is itself a $\KS$-isotypical matrix unit.  (Specifically, if we denote by $v\mapsto v^\dual$ the canonical anti-linear isomorphism from $V^\pi$ to $V^{\pi\dual}$,  then $  \overline{f(k)}  =  \overline{(w^\dual, \pi(k)v)} = (w, \pi^\dual(k) v^\dual) $.)  It follows that $p_\tau M_f P_F^\perp = (P_F^\perp M_{\overline{f}} \,p_\tau)^* =0$.  This proves that multiplication by $M_f$ is $S$-harmonically proper.  Such $\KS$-isotypical matrix units $f$ span a dense subspace of $C(\Lie{K})$.
   
\end{proof}
  
If $\mu$ and $\nu$ are weights for $\Lie{K}$, then for any $f\in\CXE[\mu]$ and $s\in\LXE[\mu]$, the product $f.s$ is in $\LXE[\mu+\nu]$, as can be readily verified from the defining equivariance property of \eref{mu-equivariance}.  Thus, for any $S\subseteq\Sigma$ the multiplication operator $M_f$ for $f\in\CXE[\mu]$ belongs to $\scrA[S](\LXE[\nu],\LXE[\nu+\mu])$.


\section{Lattice of ideals}

\begin{lemma}
\label{ordering}
If $S \subseteq T\subseteq\Sigma$, then $\scrK[T]\subseteq\scrK[S]$.
\end{lemma}

\begin{proof}
Each irreducible representation for $\Lie{K}_T$ decomposes into only finitely many irreducibles for $\KS$,  so $T$-harmonically finite operators are $S$-harmonically finite.
\end{proof}

\begin{remark}
It is not in general true that $\scrK[T]$ is an ideal $\scrK[S]$ when $S\subseteq T$.  For instance, on an infinite dimensional weight space, such as $H=\LXE[\mu] \subseteq L^2(\Lie{K})$, every bounded operator is $\emptyset$-harmonically finite.  But if $\emptyset \subsetneqq T \subsetneqq \Sigma$ then $\scr{B}(H) = \scrK[\emptyset]  \lneqq \scrK[T](H) \lneqq \scrK[\Sigma](H) =\scrK(H) $, so $\scrK[T](H)$ cannot be an ideal in $\scrK[\emptyset](H)$.
\end{remark}

In order to produce a lattice of ideals, we make the following definition.

\begin{definition}
Let $\scrA = \bigcap_{T\subseteq \Sigma} \scrA[T]$, and for each $S\subseteq\Sigma$ put
$\scrJ[S] = \scrK[S] \cap \scrA $. 
\end{definition}

Now $S\subseteq T$ implies $\scrJ[T]  \triangleleft \scrJ[S]$.

\medskip

The main result of this section is the following crucial fact about the meet operation for the lattice of ideals $\scrJ[S]$.

\medskip

\textbf{Throughout what follows we make the standing assumption that $\Lie{K}$ is a product of special unitary groups, $  \prod_{i=1}^N \SU(n_i)$,   ($n_i\geq2$).}  It is worth remarking, however, that we expect the results are true for arbitary compact semisimple groups.

\begin{theorem}
\label{meet}

If $S,T\subseteq\Sigma$ then $\scrJ[S]\cap\scrJ[T] = \scrJ[S\cup T]$.
 
\end{theorem}

The proof of Theorem \ref{meet} will occupy much of the rest of this paper.  We begin with a lemma which generalizes the notion of `essential orthotypicality' from the introduction.
 
 \begin{lemma}
 \label{asymptotic-orthogonality}
 
 Let $\Lie{K}$ be as above and $S,T\subseteq\Sigma$.  The following are equivalent.
 \begin{enumerate}
 \item On any harmonic $\Lie{K}$-space $H$, $p_\tau p_\sigma \in \scrK[S\cup T](H)$ for all $\sigma\in\KShat$, $\tau\in\KThat$.
 \item For any $\sigma\in\KShat$, $\tau\in\KThat$  and any $\epsilon>0$, there exist only finitely many irreducible representations $\pi\in\hat{\Lie{K}}_{S\cup T}$ having unit vectors $\xi\in p_\sigma V^\pi$, $\eta \in p_\tau V^\pi$ with $|\ip{\eta, \xi}|>\epsilon$.
 \item For any $\sigma\in\KShat$ and any $\epsilon>0$, there exist only finitely many irreducible representations $\pi\in\hat{\Lie{K}}_{S\cup T}$ having a unit vector $\xi\in p_\sigma V^\pi$ and a unit vector $\eta$ fixed by $\Lie{K}_T$ with $|\ip{\eta, \xi}|>\epsilon$.

 \end{enumerate}
 \end{lemma}

 \begin{proof}
 {\em(i) $\Rightarrow$ (ii)}:   Let $\sigma\in\KShat$, $\tau\in\KThat$ and $\pi\in\hat{\Lie{K}}_{S\cup T}$.  Let $U$ be the right regular representation of $\Lie{K}$ on $H=L^2(\Lie{K})$.  Note that every $\pi\in\irrep{K}_{S\cup T}$ occurs with nonzero multiplicity in $U|_{\Lie{K}_{S\cup T}}$.  Suppose $\xi\in p_\sigma p_\pi H$ and $\eta\in p_\tau p_\pi H$ are unit vectors.  
By assumption, $p_\sigma p_\tau \in \scrK[S\cup T]$, so there exists a finite set $F\subseteq \hat{\Lie{K}}_{S\cup T}$ such that  $\left\|P_F^\perp p_\sigma p_\tau\right\| <\epsilon$.  If $\pi\notin F$ then, since $P_F^\perp$ commutes with $p_\tau$ and $p_\sigma$ (Lemma \ref{commuting-projections}),
$$ 
 |\ip{\eta,\xi}| =  |\ip{P_F^\perp \eta,\xi}| = |\ip{P_F^\perp p_\tau \eta, p_\sigma \xi}|  =  |\ip{P_F^\perp p_\sigma p_\tau \eta, \xi}| <\epsilon.
$$

{\em (ii) $\Rightarrow$ (iii)}:  Immediate, by letting $\tau$ be the trivial representation of $\Lie{K}_T$.

{\em (iii) $\Rightarrow$ (i)}:   Since every irreducible representation of $\Lie{K}$ appears in the right regular representation with nonzero multiplicity, it follows that any unitary representation of $\Lie{K}$ will embed in a (possibly infinite) direct sum of copies of $L^2(\Lie{K})$.  Consequently, it suffices to prove {\em (i)} for  $H=L^2(\Lie{K})$ with the right regular representation.

Property {\em (iii)} implies that for any $\epsilon>0$ there is a finite set $F_0\subseteq\irrep{K}_{S\cup T}$ such that 
$$
\| (p_\trivT P_{F_0}^\perp)^* (p_\sigma P_{F_0}^\perp) \| = \| P_{F_0}^\perp (p_\trivT p_\sigma)\| = \| (p_\trivT p_\sigma) P_{F_0}^\perp \| < \epsilon.  
$$
Therefore, by Lemma \ref{KS-equivalent-defns}, 
\begin{equation}
\label{pp1}
  p_\sigma p_\trivT  \in \scrK[S\cup T],
\end{equation}
for every $\sigma\in\KShat$.  We want to generalize this from $\tau=\trivT$ to arbitary $\tau\in\KThat$.

Recall that for any $\tau\in\hat{\Lie{K}}_T$, $p_\tau C(\Lie{K})$ is a finitely generated projective module over $C(\XT)$ (Lemma \ref{finite-generation}).  Thus there are functions $t_1,\ldots,t_m \in p_\tau C(\Lie{K})$ such that the identity operator on  $p_\tau C(\Lie{K})$ can be factorized as
\begin{equation}
\label{identity-factorization}
  I = \sum_{i=1}^m t_i  \ipS{t_i,\slot}  =  \sum_{i=1}^m M_{t_i}\, p_\trivT \, M_{\overline{t_i}},
\end{equation}
where the latter uses Equation \eref{ES-pairing} and $M_f$ denotes multiplication by $f$.
We therefore have
$$
  p_\sigma p_\tau = \sum_{i=1}^m p_\sigma M_{t_i} p_\trivT M_{\overline{t_i}} p_\tau.
$$
Since $M_{t_i}\in\scrA[S]$ (Lemma \ref{mult-ops-in-AS}), for any $\epsilon>0$ there is a finite set $F_i\subseteq \KShat$ such that $\|p_\sigma M_{t_i} P_{F_i}^\perp\| < \epsilon/m$.  Then
\begin{equation}
\label{pp-approximation}
  \|~ p_\sigma p_\tau - \sum_{i=1}^m p_\sigma M_{t_i} P_{F_i} p_\trivT M_{\overline{t_i}} p_\tau ~ \| < \epsilon.
\end{equation}
Now, $P_{F_i} p_\trivT \in \scrK[S\cup T]$ by \eref{pp1}, while $M_{t_i}$, $M_{\overline{t_i}}$, $p_\sigma$ and $p_\tau$ are in $\scrA[S\cup T]$ (Lemmas \ref{mult-ops-in-AS} and \ref{projections-in-AS}).  Thus \eref{pp-approximation} gives an $\epsilon$-approximation of $p_\sigma p_\tau$ by an operator in $\scrK[S\cup T]$.  

\end{proof}

 \begin{proposition}
 \label{asymptotic-orthogonality-true}
 
 The equivalent properties of Lemma \ref{asymptotic-orthogonality} are true for any $S,T\subseteq\Sigma$ when $\Lie{K}$ is a product of special unitary groups.
 
 \end{proposition}

 \begin{proof}
 
 We work inductively on the size of $S\cup T$.  If $\#(S\cup T) =0$ or $1$, the result is immediate from Lemma \ref{ordering}.  So let $\#(S\cup T) =n$, and suppose we have proven the proposition for any lesser cardinalities.

Some preliminary remarks are needed.  Suppose $S\cup T\subsetneqq\Sigma$.  Decompose the Lie algebra $\lie{k}_{S\cup T}$ as
$$
 \lie{k}_{S\cup T} = \lie{z} \oplus \lie{k}',
$$
where $\lie{z}$ is the centre of $\lie{k}_{S\cup T}$ and $\lie{k}'$ is its orthogonal complement with respect to the Killing form.  Denote the corresponding connected subgroups by $\Lie{Z}$ and $\Lie{K}'$.  The group $\Lie{K}'$ is itself semisimple, and its Dynkin diagram is $S\cup T$ (with edges restricted from $\Sigma$).  For example, if $\Lie{K} = \SU(5)$ with $\Sigma = \{\alpha_1,\alpha_2,\alpha_3,\alpha_4\}$ then for $S\cup T=\{\alpha_1, \alpha_2\}$, we have
\begin{eqnarray*}
  \renewcommand{\baselinestretch}{1}\small\normalsize
     \Lie{K}' &=&  \left( \begin{array}{cccc}
        &&& \\ 
        \multicolumn{2}{c}{\raisebox{1.5ex}[0pt]{$\SU(3)$}} & \multicolumn{2}{c}{\raisebox{1.5ex}[0pt]{$ $}} \\ 
        && 1 \\
        \multicolumn{2}{c}{\raisebox{1.5ex}[0pt]{$ $}} && 1
       \end{array} \right) , \\
    \Lie{Z} &=& \left\{ \left( \begin{array}{cccc}
        &&& \\ 
        \multicolumn{2}{c}{\raisebox{1.5ex}[0pt]{$z_1I$}} & \multicolumn{2}{c}{\raisebox{1.5ex}[0pt]{$ $}} \\ 
        && \,\,z_2 \\
        \multicolumn{2}{c}{\raisebox{1.5ex}[0pt]{$ $}} && \!\!z_3 
       \end{array} \right) :\: z_1,z_2,z_3\in S^1,\: z_1^3z_2z_3=1 \right\}.
  \renewcommand{\baselinestretch}{2}\small\normalsize
\end{eqnarray*} 

By Schur's Lemma, any irreducible representation $\pi\in\irrep{K}_{S\cup T}$ is scalar on $\Lie{Z}$.  Moreover, if $\sigma\in\KShat$ and $\tau\in\KThat$ occur with nontrivial multiplicity in $\pi$, then they must agree with $\pi$ on $\Lie{Z}$.  Therefore, in checking the Property {\em(ii)} of Lemma \ref{asymptotic-orthogonality}, it suffices to consider the subgroups $\Lie{K}'$, $\KS\cap\Lie{K}'$ and $\KT\cap\Lie{K}'$ in place of $\Lie{K}_{S\cup T}$, $\KS$ and $\KT$.  We therefore assume that $S\cup T = \Sigma$.
 
We start with the case $\Lie{K}=\SU(n+1)$, with Dynkin diagram 
$$
\xymatrix@R=0pt{
  \alpha_1 & \alpha_2 && \alpha_{n} \\
   \bullet \ar@{-}[r] & \bullet \ar@{-}[r] &\cdots  \ar@{-}[r] & \bullet
  }
$$

{\em Case I}:  $S=\Sigma\setminus\{\alpha_n\}$, $T=\Sigma\setminus\{\alpha_1\}$.

This case is the computational heart of the theorem. The proof is somewhat technical so we separate it out as Lemma \ref{base-case}.
 
{\em Case II}: $S$, $T$ arbitrary with $S\cup T = \Sigma$.

Without loss of generality, suppose $\alpha_1\in S$ (otherwise interchange $S$ and $T$).  Let $S'= S\setminus\{\alpha_1\}$, $T'=T\setminus\{\alpha_1\}$.  Let $\sigma|S'$ denote the finite set of irreducible representations in $\irrep{K}_{S'}$ which occur in the restriction of $\sigma$ to $\Lie{K}_{S'}$, and similarly define $\tau|T' \subseteq \irrep{K}_{T'}$.  Then
$$
  p_\sigma p_\tau = p_\sigma P_{\sigma|S'} P_{\tau|T'} p_\tau.
$$
Now $P_{\sigma|S'} P_{\tau|T'} \in \scrK[S'\cup T'] = \scrK[\Sigma\setminus\{\alpha_1\}]$ by the inductive hypothesis.  Thus, for any $\epsilon>0$, there is a finite set $F_1\subseteq \irrep{K}_{\Sigma\setminus\{\alpha_1\}}$ such that
\begin{equation}
\label{approx1}
  \| \; p_\sigma p_\tau \:-\: p_\sigma P_{F_1} P_{\sigma|S'} P_{\tau|T'} p_\tau \: \| <\epsilon.
\end{equation}

Next consider the product $p_\sigma P_{F_1}$.   Let $S'' = S\setminus \{\alpha_n\}$ and $T'' = \Sigma\setminus\{\alpha_1,\alpha_n\}$.  As above, we let $\sigma|S'' $ denote the finite set of irreducible representations occurring in the restriction of $\sigma$ to $\Lie{K}_{S''}$, and let ${F_1}|T''$ denote the finite set of irreducible representations of $\Lie{K}_{T''}$ which occur in the restriction of any $\rho\in {F_1}$ to $T''$.  Then
$p_\sigma P_{F_1} = p_\sigma P_{\sigma|S''} P_{F_1|T''} P_{F_1}$.
Again, the inductive assumption implies $P_{\sigma|S''} P_{F_1|T''} \in \scrK[S''\cup T''] = \scrK[\Sigma\setminus\{\alpha_n\}]$, so for some finite set $F_2\subseteq\irrep{K}_{\Sigma\setminus\{\alpha_n\}}$, we have
\begin{equation}
\label{approx2}
 \|\;  p_\sigma P_{F_1} \:-\: p_\sigma P_{\sigma|S''} P_{F_1|T''} P_{F_2} P_{F_1} \:\| < \epsilon.
\end{equation}
Combining the approximations \eref{approx1} and \eref{approx2} yields
$$
 \|\; p_\sigma p_\tau \:-\:  p_\sigma P_{\sigma|S''} P_{F_1|T''} (P_{F_2} P_{F_1}) P_{\sigma|S'} P_{\tau|T'} p_\tau \:\| < 2\epsilon.
$$
But $P_{F_2} P_{F_1} \in \scrK[\Sigma]$ by Case I, and all the other projections are in $\scrA[\Sigma]$ by Lemma \ref{projections-in-AS}.  Since $\epsilon$ was arbitary, we conclude that $p_\tau p_\sigma \in \scrK[\Sigma]$.
 
 \medskip
 
Finally, we deal with the case where the Dynkin diagram of $\Sigma$ is not connected.  Let $\Sigma = \bigsqcup_{i=1}^N \Sigma_i$ be the decomposition of $\Sigma$ into connected components, which corresponds to a decomposition of $\Lie{K}$ as a product of special unitary groups 
$\Lie{K} =\prod_i \Ki$.  Irreducible representations of $\Lie{K}$ are of the form $\bigotimes_i \pi_i$, where $\pi_i\in\Kihat$.

Put $S_i=S\cup\Sigma_i$, $T_i=T\cup\Sigma_i$.  Then $\KS = \prod_i \Ki_{S_i}$, where $\Ki_{S_i}$ is the subgroup of $\Ki$ associated to the set of simple roots $S_i\subseteq\Sigma_i$.  For $\sigma\in\KShat$, we have a corresponding decomposition $\sigma=\bigotimes_i \sigma_i$, with $\sigma_i\in\Kihat_{S_i}$.  We also get $p_\sigma = \prod_i p_{\sigma_i}$.  Similarly, for $\tau\in\KThat$ we have $p_\tau = \prod_i p_{\tau_i}$, with all the analogous notation.  Since $p_{\sigma_i}$ and $p_{\tau_j}$ commute for $i\neq j$, 
$$
 p_\sigma p_\tau = \prod_i p_{\sigma_i} p_{\tau_i}.
$$
By the preceding cases, $p_{\sigma_i} p_{\tau_i} \in \scrK[\Sigma_i]$, so for any $\epsilon>0$, we can find a finite set $F_i \subseteq \Kihat$ such that $\|p_{\sigma_i} p_{\tau_i} -p_{\sigma_i} p_{\tau_i} P_{F_i}\|<\epsilon/N$, and therefore
$$
  \| p_{\sigma} p_{\tau} - p_{\sigma} p_{\tau} ( \prod_{i=1}^N P_{F_i} )  \| < \epsilon.
$$
Since $\prod_i P_{F_i} = P_F$, where $F = \{ \bigotimes_i \pi_i \st \pi_i\in F_i \} \subseteq \Khat_\Sigma$ and $P_F$ commutes with all other projections, we see that $p_\sigma p_\tau\in\scrK[\Sigma]$.

\end{proof}

 \begin{lemma}
 \label{base-case}
 Inside $\Lie{K}=\SU(n)$ ($n\geq3$), 
 let 
 \begin{eqnarray*}
   \renewcommand{\baselinestretch}{1}
     \KS &=&  \left\{ \left( \begin{array}{cccc}
        &&&0\\
        &A&&\vdots \\
        &&&0\\
        0&\cdots&0&z
       \end{array} \right)\; : \;A\in\mathrm{U}(n-1) ,\: z\in S^1,\: z(\det A) = 1 \right\},\\
  \KT &=&  \left\{ \left( \begin{array}{cccc}
        z&0&\cdots&0\\
        0&&& \\
        \vdots&&A&\\
        0&&&
       \end{array} \right) \; : \;A\in\mathrm{U}(n-1) ,\: z\in S^1, \:z(\det A) = 1 \right\},\\
  \renewcommand{\baselinestretch}{2}
\end{eqnarray*}
Let  $\sigma\in\KShat$ and $\epsilon>0$.  There are only finitely many irreducible representations $\pi$ of $\SU(n)$ which contain a unit vector $\xi\in p_\sigma V^\pi$ and $\eta\in p_\trivT V^\pi$ with $|\ip{\eta, \xi}|>\epsilon$.

\end{lemma}

The proof is a computation using Gelfand-Tsetlin bases for irreducible representations for $\SU(n)$.  We provide a quick review of this material here, which we take from the expository article of Molev \cite{Molev}.

By Weyl's unitary trick, the irreducible unitary representations of $\SU(n)$ are in correspondence with  irreducible $\CC$-linear representations of its complexified Lie algebra $\slx(n,\CC)$.  One begins by considering irreducible representations of $\glx(n,\CC)$.  The weights of $\glx(n,\CC)$ are indexed by $n$-tuples of integers, which act on the Cartan (diagonal) subalgebra by the formula
$$
  \lambda = (\lambda_1,\ldots,\lambda_n) :  \renewcommand{\baselinestretch}{1}
    \left( \begin{array}{ccc} 
     t_1 && \\ &\ddots& \\ && t_n
    \end{array} \right) \renewcommand{\baselinestretch}{2}
    \mapsto \sum_i \mu_i t_i.
$$
A weight is dominant if the entries are descending:  $\lambda_1\geq\cdots\geq \lambda_n$.  Such $n$-tuples are the highest weights of irreducible $\glx(n,\CC)$-representations.

Let $\pi_\lambda$ be the irreducible representation of $\glx(n,\CC)$ with highest weight $\lambda$.  One now considers the successive restrictions of this representation to the `upper-left' subalgebras $\lie{g}_n \supseteq \lie{g}_{n-1} \supseteq \cdots \supseteq \lie{g}_1$, where
$$
  \lie{g}_k = 
 \renewcommand{\baselinestretch}{1}
    \left( \begin{array}{ccc} 
    \\ 
    \multicolumn{2}{c}{\raisebox{1.5ex}[0pt]{$\glx(k,\CC)$}} \\
    && I 
 \end{array} \right) \renewcommand{\baselinestretch}{2}.
$$
The irreducible representations of $\lie{g}_{n-1}$ occurring in $\pi_\lambda$ are those whose highest weights $(\mu_1,\ldots,\mu_{n-1})$ satisfy the interlacing conditions
$$
   \lambda_i \geq \mu_i \geq \lambda_{i+1}   \qquad (i = 1,\ldots, n-1),
$$
and these representations each occur with multiplicity one.  Thus, a successive restriction down to $\lie{g}_1$ is specified uniquely by the rows of a {\em Gelfand-Tsetlin pattern}
$$ \Lambda = 
  \left( \begin{array}{cccccccccc}
  \multicolumn{2}{c}{\lambda_{n,1}~} &
  \multicolumn{2}{c}{~\lambda_{n,2}} &
  \multicolumn{2}{c}{\cdots\cdots\cdots} &
  \multicolumn{2}{c}{~\lambda_{n,n-1}~} &
  \multicolumn{2}{c}{~\lambda_{n,n}} \\
&\multicolumn{2}{c}{\lambda_{n-1,1}}
&\multicolumn{4}{c}{\lambda_{n-1,2} \cdots \lambda_{n-1,n-2}}
&\multicolumn{2}{c}{\!\lambda_{n-1,n-1}} \\
&&\multicolumn{2}{c}{\ddots}
&&&\multicolumn{2}{c}{\adots}\\
&&&\multicolumn{2}{c}{\lambda_{2,1}}
&\multicolumn{2}{c}{\lambda_{2,2}} \\
&&&&\multicolumn{2}{c}{\lambda_{1,1}}
\end{array} \right),
$$
satisfying
\begin{equation}
\label{interlacing}
 \lambda_{k+1,i} \geq \lambda_{k,i} \geq \lambda_{k+1,i+1} ,  \qquad (i = 1,\ldots, k-1; ~ k=1,\ldots n-1).
\end{equation}
Here, $(\lambda_{k,1},\ldots,\lambda_{k,k})$ is the highest weight of the $\glx(k,\CC)$-subrepresentation.

The resulting irreducible representations of $\lie{g}_1 \cong \glx(1,\CC)$ are one-dimensional, so choosing a nonzero vector from each will define a basis for the representation space of $\pi_\lambda$.  There is a standard choice due to \v{Z}elobenko (based on Gelfand and Tsetlin \cite{GTs}), and  we denote these basis vectors by $\xi_\Lambda$.  (We also follow the notational convention that if $\Lambda$ is an inadmissible pattern, that is it does not satisfy the interlacing conditions \eref{interlacing}, then $\xi_\Lambda=0$.)  This basis is orthogonal, but not orthonormal.  Putting $l_{k,i} = \lambda_{k,i}-i+1$, the norm of $\xi_\Lambda$ is given by
\begin{equation}
\label{norm}
  \| \xi_\Lambda \|^2 = \prod_{k=2}^n \prod_{1\leq i \leq j < k} \frac{(l_{k,i} - l_{k-1,j})!}{(l_{k-1,i} - l_{k-1,j})!}
                   \prod_{1\leq i < j \leq k} \frac{ (l_{k,i} - l_{k,j} -1)!}{(l_{k-1,i}-l_{k,j}-1)!}.
\end{equation}

The representation $\pi_\lambda$ of $\glx(n,\CC)$ is described explicitly in this basis as follows.  Let $E_{p,q}$ be the $n\times n$-matrix with all entries zero except for a $1$ in the $(p,q)$-position.   Then
\begin{eqnarray}
  \pi_\mu(E_{k,k}) \xi_{\Lambda} &=& \left( \sum_{i=1}^k \lambda_{k,i} - \sum_{i=1}^{k-1} \lambda_{k-1,i} \right) \xi_{\Lambda},
    \label{weight-operator} \\
  \pi_\mu(E_{k,k+1}) \xi_{\Lambda} &=& - \sum_{i=1}^k 
    \frac{(l_{k,i}-l_{k+1,1})\cdots(l_{k,i}-l_{k+1,k+1})}
            {(l_{k,i}-l_{k,1})\cdots\wedge\cdots(l_{k,i}-l_{k,k})}  \;\xi_{\Lambda+\delta_{k,i}},
    \label{raising-operator} \\
  \pi_\mu(E_{k+1,k}) \xi_{\Lambda} &=&  \sum_{i=1}^k 
    \frac{(l_{k,i}-l_{k-1,1})\cdots(l_{k,i}-l_{k-1,k-1})}
            {(l_{k,i}-l_{k,1})\cdots\wedge\cdots(l_{k,i}-l_{k,k})}  \;\xi_{\Lambda-\delta_{k,i}},
    \label{lowering-operator} 
\end{eqnarray}
where $\Lambda\pm\delta_{i,j}$ is the Gelfand-Tsetlin pattern obtained by adding $\pm1$ to the entry $\lambda_{i,j}$ of $\Lambda$, and the symbol $\wedge$ indicates that the zero term in the denominator should be omitted.  In particular, the Gelfand-Tsetlin vector $\xi_{\Lambda}$ is a weight vector with weight 
\begin{equation}
\label{weight}
  (s_1-s_0,s_2-s_1\ldots,s_n-s_{n-1}),
\end{equation}
where $s_k = \sum_{i=1}^k \lambda_{k,i}$ is the sum of the entries in the $k$th row, and $s_0=0$ by convention.

Now restrict the representations $\pi_\lambda$ to $\slx(n,\CC)$.  Two weights $\lambda$ and $\lambda'$ of $\glx(n,\CC)$ become equal for $\slx(n,\CC)$ iff their difference is a multiple of the trace $(1,\ldots,1)$.  Gelfand-Tsetlin patterns $\Lambda$ and $\Lambda'$ define the same basis vector if they differ by the same constant in each entry.   (Note that the Gelfand-Tsetlin formulae for $\pi_\lambda$ above are unaffected by such equivalences, although the first formula must be applied to the differences $E_{k,k} - E_{k+1,k+1} \in\slx(n,\CC)$.)

\begin{proof}[Proof of Lemma \ref{base-case}]

Suppose $\eta\in V^\pi$ is a unit vector fixed by the subgroup $\KT$.  
Let $\eta'=\pi(w)\eta$, where 
$$
 w = \left( \begin{array}{ccc}
        && 1 \\ & \adots \\ 1
       \end{array} \right)
$$
Then $\eta'$ is fixed by $\pi(w) \,\KT\, \pi(w^{-1}) = \KS$, and hence is annihilated by the complexified Lie algebra $(\lie{k}_S)_\CC$. Note that $(\lie{k}_S)_\CC$ contains the upper-left subalgebra $\lie{g}_{n-1}$, so $\eta'$ is a multiple of a Gelfand-Tsetlin vector $\xi_{\Lambda}$ with all entries of $\Lambda$ below the top row being zero  (modulo addition of a constant in each entry).     In view of the interlacing conditions \eref{interlacing}, we conclude that
$$\Lambda=
    \left( \begin{array}{cccccccccc}
  \multicolumn{2}{c}{m} &
  \multicolumn{2}{c}{0} &
  \multicolumn{2}{c}{\cdots} &
  \multicolumn{2}{c}{0} &
  \multicolumn{2}{c}{\!\!-m'} \\
&\multicolumn{2}{c}{\;0\;}
&\multicolumn{4}{c}{\;0\,\cdots\,0\;}
&\multicolumn{2}{c}{\;0\;} \\
&&\multicolumn{2}{c}{\ddots}
&&&\multicolumn{2}{c}{\adots}\\
&&&\multicolumn{2}{c}{\;0}
&\multicolumn{2}{c}{0} \\
&&&&\multicolumn{2}{c}{0}
\end{array} \right)
$$
for some $m,m'\geq0$.
Moreover, $\eta'$ is of weight zero, since $(\lie{k}_S)_\CC$ contains the Cartan subalgebra $\lie{h}=\lie{t}_\CC$, so by \eref{weight}, $m=m'$.  In particular, the representation $\pi$ has highest weight of the form $\lambda = (m,0,\ldots,0,-m)$.

With $\pi=\pi_\lambda$ thus specified, let $\eta_m$ be the $\KT$-fixed unit vector for $\pi_\lambda$.  It has weight zero and is annihilated by $\pi_\lambda(E_{k,k+1})$ for $k=2,\ldots,n-1$.  The zero-weight space of $\pi_\lambda$ is spanned by the Gelfand-Tsetlin vectors for patterns with zero row-sums,
\begin{equation}
\label{inductive-GTs-pattern}
   \left( \begin{array}{cccccccccc}
  \multicolumn{2}{c}{m} &
  \multicolumn{2}{c}{0} &
  \multicolumn{2}{c}{\cdots} &
  \multicolumn{2}{c}{0} &
  \multicolumn{2}{c}{\;\;-m} \\
&\multicolumn{2}{c}{\!m_{n-1}\!\!\!}
&\multicolumn{4}{c}{\!\!\!\!\!\!\!\!0\cdots0\!\!}
&\multicolumn{2}{c}{\!\!\!\!\!-m_{n-1}} \\
&&\multicolumn{2}{c}{\ddots}
&&&\multicolumn{2}{c}{\adots}\\
&&&\multicolumn{2}{c}{m_2}
&\multicolumn{2}{c}{\!\!\!-m_2} \\
&&&&\multicolumn{2}{c}{0}
\end{array} \right)
\end{equation}
We will denote such a Gelfand-Tsetlin pattern by $\Lambda(M)$, where ${M}$ is the $n$-tuple $(m_n,m_{n-1},\ldots,m_2,m_1)$ with $m=m_n\geq m_{n-1} \geq \cdots \geq m_2\geq m_1=0$.

\begin{claim} 
\label{claim}
$\displaystyle \left|\bigip{\eta_m,\frac{\xi_{\Lambda(M)}}{\|\xi_{\Lambda(M)}\|}}\right| = \frac{1}{{{m+n-2} \choose {n-2}}}  \left( \frac{\prod_{k=2}^{n-1} (2m_k+k-1) } {{(n-2)!}} \right)^\half. $
\end{claim}

The important point here is that for fixed values of $m_{n-1},\ldots,m_2$, these inner products tends to zero as $m\to\infty$.  From this, Lemma \ref{base-case} follows.  For suppose $\sigma$ is an irreducible representation of $\KS$.  If $\sigma$ does not have highest weight of the form $(q,0,\ldots,0,-q')$ then it occurs with zero multiplicity in the representations $\pi_\lambda$ which have $\KT$-fixed vectors, and if $q\neq q'$ then none of its vectors have zero weight, so it is orthogonal to $\eta_m$.  On the other hand, if $\sigma$ does have highest weight $(q,0,\ldots,0,-q)$, then any unit vector $\xi\in p_\sigma V^{\pi_\lambda}$ is a linear combination of the vectors $\xi_{M}$ with $m_{n-1} = q$.  There are at most $q^{n-2}$ such vectors, regardless of $m$, and Claim \ref{claim} shows that they are all asymptotically orthogonal to $\eta$ as $m\to\infty$.

Let us prove Claim \ref{claim}.  For $1<k<n$, let $M\pm e_k$ denote the $n$-tuple $(m_{n},\ldots, m_k\!\pm\!1, \ldots, m_1)$.  Recall that we use $\Lambda \pm \delta_{i,j}$ to denote the pattern obtained by adding $\pm1$ to the $(i,j)$-entry of $\Lambda$.  Note that, $\Lambda(M)+\delta_{k,j}$ does not satisfy the interlacing conditions \eref{interlacing} unless $j=1$ or $k$.  Note also that $\Lambda(M)+\delta_{k,k} = \Lambda(M-e_k)+\delta_{k,1}$.

Write $\eta_m$ in the Gelfand-Tsetlin basis for $\pi_\lambda$:
$$\eta_m = \sum_M a_M \xi_{\Lambda(M)}.$$
By formula \eref{raising-operator},
\begin{eqnarray*}
  \lefteqn{\pi_{\lambda}(E_{k,k+1}) \xi_{\Lambda(M)} } \\
      &=& -\frac{(m_k-m_{k+1}) \:\left(\prod_{j=1}^{k-1} (m_k+j)\right)\: (m_k+m_{k+1}+k)}
           { \:\left(\prod_{j=1}^{k-2} (m_k+j)\right)\: (2m_k+k-1)} \;\xi_{\Lambda(M)+\delta_{k,1}} \\
      && \quad -\frac{ (-m_k-m_{k+1}-k+1) \:\left(\prod_{j=0}^{k-2} (-m_k-j)\right)\: (-m_k+m_{k+1}+1)}
                                    {(-2m_k-k+1)\:\left(\prod_{j=1}^{k-2} (-m_k-j)\right)\: }  \;\xi_{\Lambda(M)+\delta_{k,k}}  \\
      &=& \frac{(m_{k+1}-m_k)(m_{k+1}+m_k+k)(m_k+k-1)}{(2m_k+k-1)} \;\xi_{\Lambda(M)+\delta_{k,1}} \\
      && \quad +\frac{(m_{k+1}-m_k+1)(m_{k+1}+m_k+k-1)\,m_k}{(2m_k+k-1)}\; \xi_{\Lambda(M-e_k)+\delta_{k,1}},
\end{eqnarray*}
for $k=2,\ldots,n-1$.  Comparing the coefficients of $\xi_{\Lambda(M)+\delta_{k,1}}$ in the equation
$$
  \pi_\lambda(E_{k,k+1})\, \eta_m = \sum_M a_M \pi_\lambda(E_{k,k+1})\xi_{\Lambda(M)} = 0,
$$
we see that
\begin{multline*}
  \frac{(m_{k+1}-m_k)(m_{k+1}+m_k+k)(m_k+k-1)}{(2m_k+k-1)}  \;a_M \\
    + \frac{(m_{k+1}-m_k)(m_{k+1}+m_k+k)(m_k+1)}{(2m_k+k+1)} \;a_{M+e_k} = 0,
\end{multline*}
so that
\begin{equation}
\label{inductive-reduction}
   a_{M+e_k} = -\frac{(m_k+k-1)}{(m_k+1)}\cdot \frac{(2m_k+k+1)}{(2m_k+k-1)} \, a_M,
\end{equation}
for $k=2,\ldots,n-1$.
We can use \eref{inductive-reduction} to reduce each of the entries $m_2,\ldots,m_{k-1}$ in turn, resulting in
\begin{eqnarray}
   a_M &=& \pm \left( \prod_{k=2}^{n-1} \prod_{i=0}^{m_k-1} \frac{(i+k-1)}{(i+1)} \cdot \frac{(2i+k+1)}{(2i+k-1)} \right) \, a_{(m,0,\ldots,0)} \nonumber \\
         &=& \pm \frac{1}{(n-2)!} \left(\prod_{k=2}^{n-1} \frac{(m_k+k-2)!}{m_k!} \, (2m_k+k-1)\right) \, a_{(m,0,\ldots,0)}. 
         \label{aM}
\end{eqnarray}

We now compute $\|\xi_{\Lambda(M)}\|$ by Equation \eref{norm}. This is straightforward but tedious.  The $k=2$ term in \eref{norm} is ${m_2!}\frac{(2m_2)!}{m_2!}=(2m_2)!$.  For $3\leq k \leq n$, the terms with $i=1$ give
\begin{multline*}
  \frac{(m_k-m_{k-1})! \:\left( \prod_{j=1}^{k-3} (m_k+j)!\right)\: (m_k+m_{k-1}+k-2)!} 
          {0! \:\left( \prod_{j=1}^{k-3} (m_{k-1}+j)! \right)\: (2m_{k-1}+k-2)!}  \\
\times  \frac{m_k! \:\left( \prod_{j=1}^{k-3} (m_k+j)! \right)\: (2m_k+k-2)!} 
          {m_{k-1}! \:\left( \prod_{j=1}^{k-3} (m_{k-1}+j)! \right)\: (m_k+m_{k-1}+k-2)!}  \\
    = (m_k-m_{k-1})! \cdot\left( \prod_{j=0}^{k-3} \frac{ (m_k+j)!}{(m_{k-1}+j)!} \right)^2 \cdot
       \frac{m_{k-1}!}{m_k!} \cdot  \frac{(2m_k+k-2)!}{(2m_{k-1}+k-2)!} ;
\end{multline*}
the terms with $ 1<i<k-1$ are all $1$; and the terms with $i=k-1$ give 
$$
 \frac{  m_{k-1}!\,m_k!}{(m_k-m_{k-1})!}.
$$
Thus,
\begin{eqnarray}
  \| \xi_{\Lambda(M)} \|^2 &=&
     \prod_{k=3}^n \left[ \left( \prod_{j=3}^{k} \frac{(m_k+j-3)!}{(m_{k-1}+j-3)!} \right)^2 \cdot
        {m_{k-1}!}^2 \cdot \frac{(2m_k+k-2)!}{(2m_{k-1}+k-2)!}  \right] \,(2m_2)! \nonumber \\
   &=& \left( \prod_{j=3}^n \prod_{k=j}^n \frac{(m_k+j-3)!}{(m_{k-1}+j-3)!} \right)^2 \cdot
               \left( \prod_{k=3}^n m_{k-1}! \right)^2 \cdot \nonumber \\
   && \hspace{2cm}            
               \left( \prod_{k=3}^n \frac{(2m_k+k-2)!}{(2m_{k-1}+k-3)!} \, \frac{1}{(2m_{k-1}+k-2)} \right)(2m_2)! \nonumber \\
   &=& \left( \prod_{j=3}^n \frac{(m_n+j-3)!}{(m_{j-1}+j-3)!} \right)^2 \cdot
              \left(\prod_{k=3}^{n} m_{k-1}! \right)^2 \nonumber \\
   && \hspace{2cm} 
               (2m_n+n-2)! \, \prod_{k=3}^{n} \frac{1}{(2m_{k-1}+k-2)} \nonumber \\
   &=& C(m)\, \prod_{k=2}^{n-1} \left(\frac{m_k!^2}{(m_k+k-2)!^2} \cdot \frac{1}{(2m_k+k-1)}\right)  ,
     \label{xiM}
\end{eqnarray}
where
$$
  C(m) = \left( \prod_{j=3}^{n} (m+j-3)! \right)^2 \, (2m+n-2)!.
$$
Combining \eref{aM} and \eref{xiM}, we have 
\begin{eqnarray*}
1 \;=\; \| \eta_m \|^2 
      &=& \sum_{M} |a_M|^2 \| \xi_{\Lambda(M)} \|^2 \\
      &=& \frac{ C(m)\,{a_{(m,0,\ldots,0)}}^2 }{(n-2)!^2} 
             \sum_{m\geq m_{n-1} \geq \cdots \atop \hspace{3ex} \cdots \geq m_3 \geq m_2 \geq 0} 
                    \prod_{k=2}^{n-1} (2m_k+k-1) 
\end{eqnarray*}
A combinatorial identity (Lemma \ref{combinatorial-identity} below) shows that this equals
$$
   \frac{ C(m)\,{a_{(m,0,\ldots,0)}}^2 }{(n-2)!^2} \; (n-2)! \, {m+n-2 \choose n-2}^2,
$$
so
$$
  { C(m)^{\half}\,{a_{(m,0,\ldots,0)}} }= \frac{(n-2)!^\half}{{{m+n-2} \choose {n-2}}}
$$
We therefore have
\begin{eqnarray*}
  \left| \bigip{\eta_m , \frac{\xi_{\Lambda(M)}}{\|\xi_{\Lambda(M)}\|} } \right|
    &=& a_M \| \xi_{\Lambda(M)} \| \\
    &=& \frac{ C(m)^{\half}\,{a_{(m,0,\ldots,0)}} }{(n-2)!} \prod_{k=2}^{n-1} (2m_k+k-1)^\half \\
    &=& \frac{1}{{{m+n-2} \choose {n-2}}}  \left( \frac{\prod_{k=2}^{n-1} (2m_k+k-1) } {{(n-2)!}} \right)^\half,
\end{eqnarray*}
as claimed.
\end{proof}

We needed the following combinatorial identity.
 
\begin{lemma}
\label{combinatorial-identity}

\begin{equation}
\label{identity}
\sum_{m\geq m_{n-1} \geq \cdots \atop \hspace{3ex} \cdots \geq m_3 \geq m_2 \geq 0} 
                     \prod_{k=2}^{n-1} (2m_k+k-1) = (n-2)! \, {m+n-2 \choose n-2}^2. 
\end{equation}

\end{lemma}

\begin{proof}

Firstly, the identity
\begin{equation}
\label{identity1}
  \sum_{i=0}^m (2i+p+1) {i+p \choose p}^2 = (p+1) {m+p+1 \choose p+1}^2
\end{equation}
is proven by induction on $m$.  Now equation \eref{identity} is proven by induction on $n$, as follows.  If $n=3$, then \eref{identity} is
$$
  \sum_{m_2=0}^m (2m_2+1) = (m+1)^2,
$$
which is just \eref{identity1} with $p=0$.  For $n>3$, write the left-hand side of \eref{identity} as
\begin{multline*}
  \sum_{m_{n-1}=0}^m \left( (2m_{n-1}+n-2) \cdot \hspace{-3ex}
    \sum_{m_{n-1} \geq m_{n-2} \geq \cdots \atop \hspace{3ex} \cdots \geq m_3 \geq m_2 \geq 0} 
                     \prod_{k=2}^{n-2} (2m_k+k-1) \right) \\
    =  \sum_{m_{n-1}=0}^m (2m_{n-1}+n-2)\, (n-3)! {m_{n-1}+n-3 \choose n-3}^2,
\end{multline*}
by the inductive hypothesis.  Applying \eref{identity1} with $p=n-3$, gives the result.

\end{proof}

The above proof is unquestionably very computational.  It would be extremely satisfying to have a proof of Proposition \ref{asymptotic-orthogonality-true} which is more geometric in nature, especially given the expected wide generality of the result, as suggested in Question \ref{conjecture}.

\begin{corollary}
\label{projections-in-JS}
 
Let $\Lie{K}$ be a product of special unitary groups.  For any $S\subseteq\Sigma$, the isotypical projections $p_\sigma$  (\,$\sigma\in\KShat$) are in $\scrJ[S]$.
 
\end{corollary}
 
\begin{proof}
 
We need that $p_\sigma\in\scrA[T]$ for any $T\subseteq \Sigma$.  Let $B$ be $T$-harmonically proper.  Then $B = P_F B$ for some finite set $F\subseteq\KThat$.  By Theorem \ref{asymptotic-orthogonality-true}, $p_\sigma P_F \in \scrK[S\cup T] \subseteq \scrK[T]$, so $p_\sigma B = (p_\sigma P_F) B \in \scrK[T]$.  This shows that $p_\sigma$ multiplies $\scrK[T]$ on the left.  A similar argument on the right shows that $p_\tau$ is a two-sided multiplier, and Lemma \ref{AS-equivalent-defns} applies.

\end{proof}

\begin{proof}[Proof of Theorem \ref{meet}]

Suppose $A\in\scrJ[S]$ and $B\in\scrJ[T]$.  Use Lemma \ref{KS-equivalent-defns} to approximate these by $AP_{F_1}$ and $P_{F_2} B$ for some finite sets of irreducibles $F_1\subseteq \KShat$ and $F_2\subseteq \KThat$.  Since $P_{F_1}P_{F_2}\in\scrK[S\cup T]$ and all of $A$, $B$, $P_{F_1}$ and $P_{F_2}$ are in $\scrA$, the result follows.

\end{proof}

 \section{Products of longitudinal pseudodifferential operators}
 
 \begin{lemma}
 \label{groupoid-algebra-in-JS}
 
Let $\Lie{K}$ be a product of special unitary groups.  Let $E$ be an equivariant vector bundle over $\scrX$ and let $S\subset\Sigma$.   For any $-p<0$, the longitudinal pseudodifferential operators $\Psi^{-p}_{\FS}(E)$ of order $-p$, tangent to the fibration $\scrX\to\XS$, are contained in $\scrJ[S]$.

\end{lemma}

\begin{proof}
From Section \ref{longitudinal-PsiDO}, $\overline{\Psi^{-p}_{\FS}}(E)$ consists of compact Hilbert module operators on $\ESXE$.  
By decomposing $E$ into a direct sum of $\Lie{K}$ homogeneous line bundles, we are reduced to considering  a compact operator $A$ from $\ESXE[\mu]$ to $\ESXE[\nu]$ for some $\mu,\nu\in\Lambda_W$.  It suffices to work with rank-one operators,
$$
   A  = t_2.\ip{ t_1, \slot}_{\CXS}  = M_{t_2} p_\trivS M_{\overline{t_1}},
$$
for  $t_1\in\ESXE[\mu]$, $t_2\in\ESXE[\mu]$.    But since $p_\trivS\in\scrJ[S]$ (Lemma \ref{projections-in-JS}) and multiplication operators are in $\scrA$ (Lemma \ref{mult-ops-in-AS}), we are done. 
\end{proof}

 Combining this with Theorem \ref{meet}, we have proven the following.

 \begin{corollary}
 
Let $\Lie{K} = \SU(n)$ so that $\scrX$ is the complete flag variety for $\CC^n$.  Let $S_1,\ldots,S_N\subseteq\Sigma$ with $\bigcap_i S_i = \Sigma$, and let $A_i$ be a longitudinal pseudodifferential operator of negative order tangent to the fibration $\scrX \to \scrX_{S_i}$.  Then the product $\prod_i A_i$ is a compact operator on $L^2(\Lie{K})$.

\end{corollary}

\bibliographystyle{alpha}

\bibliography{psidos-on-flag-varieties}

\end{document}